\documentclass[12pt,a4paper]{amsart}
\usepackage{}
\usepackage{txfonts}
\usepackage{amsfonts}
\usepackage{enumerate}

\usepackage{amsmath,amssymb,xspace,amsthm}
\newtheorem{theorem}{Theorem}

\newtheorem{example}[theorem]{Example}
\newtheorem{lemma}[theorem]{Lemma}
\newtheorem{proposition}[theorem]{Proposition}
\newtheorem{corollary}[theorem]{Corollary}

\numberwithin{equation}{section}

\def\g{\gamma}

\def\span{\operatorname{span}}

\newcommand{\C}{\ensuremath{\mathbb C}\xspace}
\renewcommand\sl{\mathfrak{sl}}

\newcommand{\Z}{\ensuremath{\mathbb{Z}}\xspace}
\newcommand{\N}{\ensuremath{\mathbb{N}}\xspace}

\renewcommand{\H}{\mathcal{H}}

\newcommand{\tg}{\tilde{g}}

\def\g{\mathfrak{g}}
\def\tg{\widetilde{\g}}

\newcommand{\tH}{\tilde{\mathcal{H}}}

\newcommand{\supp}{\ensuremath{\operatorname{supp}}\xspace}
\newcommand{\Ind}{\ensuremath{\operatorname{Ind}}\xspace}

\newcommand{\ad}{\operatorname{ad}\xspace}

\renewcommand{\phi}{\varphi}

\begin{document}
\title[Simple modules over conformal Galilei algebras]{On simple modules over \\ conformal Galilei algebras}
\author{Rencai L{\"u}, Volodymyr Mazorchuk and Kaiming Zhao}
\date{\today}

\begin{abstract}
We study irreducible representations of two classes of conformal Galilei algebras in $1$-spatial dimension.  
We construct a functor which transforms simple modules with nonzero central charge over the Heisenberg subalgebra
into simple modules over the conformal Galilei algebras. This can be viewed as an analogue of 
oscillator representations. We use oscillator representations to describe the structure of simple highest 
weight modules over conformal Galilei algebras. We classify simple weight modules with finite dimensional
weight spaces over finite dimensional Heisenberg algebras and use this classification and
properties of oscillator representations to classify simple weight modules with finite dimensional
weight spaces over conformal Galilei algebras.
\end{abstract}
\vspace{2mm}
\maketitle

\vspace{2mm}

\noindent{\bf 2010 Math. Subj. Class.:} 17B10, 17B81, 22E60
\vspace{2mm}

\section{Introduction and description of results}\label{s0}

Galilei groups and their Lie algebras  are important objects in theoretical
physics and attract a lot of attention in related mathematical areas, see for example \cite{AIK, CZ, GI1,
GI2, HP, NOR}. The conformal extension $\tg^{(l)}$ of the Galilei algebra is parameterized by 
$l\in\frac{1}{2}\mathbb{N}$ and is called the {\em $l$-conformal Galilei algebra}, see \cite{NOR}. 
The instance of $l=1/2$, known in the literature as the Schr{\"o}dinger algebra, turns up in a variety
of physical systems, see for example \cite{ADD, DDM, HU}. Various classes of representations of the
Schr{\"o}dinger algebra were studied and classified in \cite{DDM,WZ,D}. Motivated by recent investigations 
of the non-relativistic version of the AdS/CFT correspondence, conformal Galilei algebras with $l>1/2$ have
recently attracted considerable attention, see \cite{AIK,LSW1,LSW2,BG,DH}. In this paper, we study 
representations of the conformal Galilei algebras $\g^{(l)}$ for $l\in\frac{1}{2}\mathbb{N}$ and 
$\tg^{(l)}$ for $l\in\N-\frac{1}{2}$. Let us start by defining these algebras.

We denote by $\mathbb{Z}$, $\mathbb{Z}_+$, $\N$, and $\mathbb{C}$ the sets of  all integers, nonnegative 
integers, positive integers, and complex numbers, respectively. For a Lie algebra
$\mathfrak{a}$ we denote by $U(\mathfrak{a})$ the universal enveloping algebra of $\mathfrak{a}$.

For $l\in \N-\frac{1}{2}$, the conformal centrally extended Galilei algebra $\tg=\tg^{(l)}$ is the Lie algebra 
with the basis $\{e, h, f, p_{k},z\,\,|\,\,k=0,\ldots,2l\}$ and the Lie bracket given by:
\begin{displaymath}
[h,e]:=2e,\,\,\,[h,f]:=-2f,\,\,\, [e,f]:=h,
\end{displaymath}
\begin{displaymath}
[h,p_{k}]:= 2(l-k)p_{k},\,\,\,[e,
p_{k}]:=kp_{k-1},\,\,\,[f,p_k]:=(2l-k)p_{k+1},
\end{displaymath}
\begin{displaymath}
[z,\g^{(l)}]:=0,
\end{displaymath}
\begin{displaymath}  
[p_{k},p_{k'}]:=\delta_{k+k',2l}(-1)^{k+l+\frac{1}{2}}k!(2l-k)!
z,\quad k,k'=0,1,\ldots,2l.
\end{displaymath}
Denote  by $ \g=\g^{(l)}$ the quotient Lie algebra $\tg^{(l)}/\C z $ and note that the Lie algebras $\g^{(l)}$ 
can be defined for all $l\in\N$ by setting $z=0$ in the above. The Lie algebras $\g^{(l)}$ for $l\in\N$ do not 
have nontrivial central extensions. The algebras $\g^{(l)}$ for $l\in\frac12\N$ are called {\em centerless 
conformal Galilei algebras}. In this paper we study both $\g^{(l)}$ for $l\in\frac12\N$  and 
$\tg^{(l)}$ for $l\in\N-\frac12$.

The conformal Galilei algebra $\tg^{(l)}$ has the finite dimensional Heisenberg subalgebra:
$$\tH=\tH^{(l)}:=\span\{p_{k},z\,\,\,|\,\,\,k=0,1,\ldots,2l\}.$$
The algebra  $\tg^{(l)}$ is $\Z$-graded with respect to the
adjoint action of $h$.  For $i\in\Z$ we set 
$$\tg_i:=\{x\in\tg\,\,|\,\,[h,x]=ix\},\quad\tg_0:=\C h+\C z,\quad\tg_{\pm}:=\bigoplus_{i\in \pm \N} \tg_i.$$
For $l\in \N$ we have a Cartan subalgebra $\g_0:=\C h+\C p_l$ in $\g^{(l)}$ and 
an abelian subalgebra $\H=\span\{p_{k}\,\,\,|\,\,\,k=0,1,\ldots,2l\}$ in $\g^{(l)}$.

Both $\tg$ and $\g$ have an $\sl_2$-subalgebra spanned by $e,f,h$. Let $\mathfrak{a}$ be any of the following
Lie algebras: $\sl_2$, $\tg$, $\g$, or $\C h+\tH$. An $\mathfrak{a}$-module $V$ is called a weight
module if $h$ acts diagonally on $V$, i.e. $$ V=\bigoplus_{\lambda\in
\C} V_{\lambda},$$ where $V_{\lambda}:=\{v\in V \,\,|\,\, h v=\lambda v\}$.  Denote
$\supp(V):=\{\lambda\in \C|V_{\lambda}\ne 0\}$.  

For any $(\dot{z}, \dot{h})\in \C^2$,  we have the one dimensional
$\tg_+\oplus \tg_0$-module $\C w$ with $d w=\dot{d} w$, $z w=\dot{z}
w$ and $\tg_+ w=0$. Using it we define, as usual, the Verma module 
$$M_{\tg}(\dot{z},\dot{h})=\Ind_{\tg_0+\tg_+}^{\tg} \C w.$$
The module $M_{\tg}(\dot{z},\dot{h})$ is generated by a
highest weight vector (the lift of $w$) of highest weight $\dot{h}$.
Denote by $\bar{M}_{\tg}(\dot{z},\dot{h})$ the unique simple
quotient of $M_{\tg}(\dot{z},\dot{h})$. Similarly we may
define the Verma modules $M_\g(\dot{p}_{l},\dot{h})$ for $l\in \N$,
$M_{\sl_2}(\dot{h})$, $M_{\tH}(\dot{z})$ for $l\in \N-1/2$, and the
corresponding simple quotients $\bar{M}_\g(\dot{p}_{l}, \dot{h})$,
$\bar{M}_{\sl_2}(\dot{h})$ and $\bar{M}_{\tH}(\dot{z})$. Verma modules have the 
usual universal property that each module generated by a highest weight vector is 
a quotient of  some Verma module. Similarly one defines lowest weight module.

We start with a brief overview of known results.
It is well-known that the Verma module  $M_{\sl_2}(\dot{h})$ is
simple if and only if $\dot h\notin \Z^+$ (see e.g.  \cite[Chapter~3]{M}),  
while $M_{\tH}(\dot{z})$ is simple if and only if $\dot z\ne0$ (see  e.g. \cite{KR}).
All simple modules over $\sl_2$ and $\tH^{(\frac{1}{2})}$ were classified in 
\cite{B} (up to description of irreducible elements
of some non-commutative Euclidean rings). An explicit description of simple lowest (or
highest) weight $\tg^{(l)}$-modules for small $l$ was given in
\cite{AI, DDM, Mr}. Furthermore, an explicit description of simple lowest
(or highest) weight $\tg^{(l)}$-modules over all $\tg^{(l)}$ was obtained in
\cite{AIK}. Very recently, a classification of simple weight modules
with finite dimensional weight spaces over the Shr{\"o}dinger algebra $\tg^{(1/2)}$ 
was given in \cite{D}.

Next we briefly describe the results of this paper. Our fist observation is
the following construction of what is natural to call {\em oscillator representations} 
for $\tg^{(l)}$. For $l\in \N-\frac{1}{2}$ denote by $U(\tH^{(l)})_{(z)}$ the localization of
$U(\tH^{(l)})$ with respect to the multiplicative subset $\{z^i\vert i\in\Z_0\}$.

\begin{theorem} \label{thm-1}
Let $l\in \N-\frac{1}{2}$.
\begin{enumerate}[$($i$)$]
\item\label{thm-1.1} 
There is a unique algebra homomorphism $\Phi_l:U(\tg^{(l)})\to U(\tH^{(l)})_{(z)}$ which
is the identity on $U(\tH^{(l)})$ and such that 
\begin{displaymath}
\Phi_l(e)=E:=\frac{1}{{z}}\left(\sum_{k=1}^{2l}(-1)^{k+l+\frac{1}{2}}
\frac{(l-k)}{(k-1)!(2l-k)! }p_{k-1}p_{2l-k}\right),
\end{displaymath}
\begin{displaymath}
\Phi_l(f)=F:=\frac{1}{{z}}\left(\sum_{k=1}^{2l}(-1)^{k+l
-\frac{1}{2}}\frac{(l-k)}{(k-1)!(2l-k)!
}p_{k}p_{2l-k+1}\right),
\end{displaymath}
\begin{displaymath}
\Phi_l(h)=H:=\frac{2}{{z}}\left(\sum_{k=0}^{l-\frac{1}{2}}
(-1)^{k+l-\frac{1}{2}}\frac{l-k}{k!(2l-k)!}p_{2l-k}p_k\right)-\frac{(l+\frac{1}{2})^2}{2}.
\end{displaymath}
\item\label{thm-1.2} 
Composition of induction from $U(\tH^{(l)})$ to $U(\tH^{(l)})_{(z)}$ with pulling back along
$\Phi$ defines a functor from the category of $\tH^{(l)}$-modules to the category of $\tg^{(l)}$-modules 
which annihilates all simple $\tH^{(l)}$-modules with zero central charge and sends simple 
$\tH^{(l)}$-modules with nonzero central charge to simple $\tg^{(l)}$-modules.
\end{enumerate}
\end{theorem}

Theorem~\ref{thm-1} is inspired by  generalized oscillator representations for the Heisenberg-Virasoro
algebra constructed in \cite[Section~3.2]{LZ1}.
If $V$ is a simple $\tH^{(l)}$-module with nonzero action of $z$, we will denote the image of
$V$ under the functor described in Theorem~\ref{thm-1}\eqref{thm-1.2} by  $V^{\tg}$.
Any $\sl_2$-module $V$ becomes a $\tg^{(l)}$-module by setting $\H\cdot V=0$. 
The resulting  $\tg^{(l)}$-module will be also denoted by $V^{\tg}$ (we will make sure that
this abuse of notation will not create any confusion). 

Consider $\tg^{(l)}$ for $l\in \N-\frac{1}{2}$. Then for every $\dot{h}\in\C$ we have
(using the uniqueness of simple highest weight modules) that
$\bar{M}_{\tg^{(l)}}(0,\dot{h})\cong \bar{M}_{\sl_2}(\dot{h})^{\tg}$, in particular,
$\tH \bar{M}_{\tg^{(l)}}(0,\dot{h})=0$. 
All highest weight $\tg^{(l)}$-modules with nonzero central charge
were explicitly described in \cite{AIK}. These modules can now be realized in a totally
different way as tensor products of certain $\sl_2$-modules and oscillator modules as follows:

\begin{theorem}\label{thm-2} 
Let $l\in \N-\frac{1}{2}$ and $\tg=\tg^{(l)}$. For any $\dot{h}\in \C$ and $\dot{z}\in \C^*$ we have
$$M_{\tg}(\dot{z},\dot{h})\cong M_{\tH}(\dot{z})^{\tg}\otimes
M_{sl_2}(\dot{h}+{\text{\tiny$\frac{1}{2}$}}(l+{\text{\tiny$\frac{1}{2}$}})^2)^{\tg},$$
$$\bar{M}_{\tg}(\dot{z},\dot{h})\cong
{M}_{\tH}(\dot{z})^{\tg}\otimes
\bar{M}_{sl_2}(\dot{h}+{\text{\tiny$\frac{1}{2}$}}(l+{\text{\tiny$\frac{1}{2}$}})^2)^{\tg}.$$
Hence the Verma module $M_{\tg}(\dot{z},\dot{h})$ is simple
if and only if $\dot{h}+\frac{(l+\frac{1}{2})^2}{2}\not\in \Z_+$.
\end{theorem}

Theorem~\ref{thm-2} provides a very clear description of simple highest weight $\tg^{(l)}$-modules 
for  $l\in \N-\frac{1}{2}$. Furthermore, Theorem~\ref{thm-2} generalizes to the following result
which even covers non-weight modules:

\begin{theorem}\label{thm-3}  
Let $l\in \N-\frac{1}{2}$ and $\tg=\tg^{(l)}$.
\begin{enumerate}[$($i$)$]
\item\label{thm-3.1}
If $M$ is a simple $\tH^{(l)}$-module with nonzero central
charge, and $N$ is a simple $\sl_2$-module, then $M^{\tg}\otimes
N^{\tg}$ is a simple $\tg$-module.
\item\label{thm-3.2}
Suppose that $ V$ is a simple $\tg$-module with nonzero central
charge $\dot{z}$. If $\mathrm{Res}^{\tg}_{\tH^{(l)}}V$ contains a simple submodule
$M$, then $V\cong M^{\tg}\otimes N^{\tg}$ for some simple
$\sl_2$-module $N$.
\end{enumerate}
\end{theorem}

To describe simple highest weight $\g^{(l)}$-modules for $l\in \N$, we need more notation. 
Let   $\dot p_l, \dot h\in\C$ and $$\mathfrak{n}=\span\{e,h,p_i\,|\,i=0,\ldots, 2l\},\qquad F_1=\C[x].$$
Define an $\mathfrak{n}$-module structure on the Fock space $F_1$ as follows:
\begin{equation}\label{eqaa1}
e=\frac{\partial}{\partial x},\,\, h=\dot h-\frac{2x\partial}{\partial x},\,\,p_l=\dot 
p_l,\,\, p_0=p_1=...=p_{l-1}=0,
\end{equation}
\begin{equation}\label{eqaa2}
p_{l+k}=\frac{(l+k)!x^k}{k!}, \text{for} \,\,\,k=1, 2, ..., l.
\end{equation}
Now we can describe all simple highest weight $\g^{(l)}$-modules
(note that in the case $l=1$ this statement  is proved in \cite{W}):

\begin{theorem}\label{highestN} 
Let $l\in \N$,  and $\g=\g^{(l)}$.
\begin{enumerate}[$($i$)$]
\item\label{highestN.1} If $\dot{p}_{l}=0$, then $\bar{M}_\g(0, \dot{h})\cong
\bar{M}_{\sl_2}(\dot{h})^{\g}$.
\item\label{highestN.2} If $\dot{p}_{l}\ne 0$, then $ \bar{M}_\g(\dot{p}_{l},
\dot{h})_{\dot{h}-2i} \cong \Ind_{\mathfrak{n}}^{\g^{(l)}}F_1$. 
\end{enumerate}
\end{theorem}

For any $\dot{z}\in \C^*$ and $a\in \C\setminus \Z$, the
$\tH^{(\frac{1}{2})}$-module structure $D(a,\dot{z})$ on $\C[x,x^{-1}]$ is defined
as follows: for $i\in\Z$ we set
\begin{equation}\label{eqaa3} 
p_1 x^i=x^{i+1},\quad p_0 x^i=-\dot{z}(a+i)x^{i-1},\quad z x^i=\dot{z} x^i.
\end{equation}
The next two theorems give a complete classification of simple weight
modules which have a nonzero finite dimensional weight space over the algebras
$\g^{(l)}$ for $l\in \frac{1}{2} \N$ and $\tg^{(l)}$ for $l\in \N-\frac{1}{2}$, respectively.

\begin{theorem} 
\label{thm-4-1}
Let $l\in \N$, $\g=\g^{(l)}$, and $V$ be any simple weight
$\g$-module with a finite dimensional nonzero weight space. Then  one of
the following holds:
\begin{enumerate}[$($i$)$]
\item\label{thm-4-1.1} $V$ is isomorphic to $N^{\g}$ for some simple weight $\sl_2$
module $N$;
\item\label{thm-4-1.2} $V$ is a highest or lowest weight module.
\end{enumerate}
\end{theorem}

\begin{theorem}\label{thm-4} 
Let $l\in \N-\frac{1}{2}$, $\tg=\tg^{(l)}$, and $V$ be any simple weight
$\tg$-module with a finite dimensional nonzero weight space.
\begin{enumerate}[$($i$)$]
\item\label{thm-4.1} If $V$ has zero central charge, the  $V$ is isomorphic to 
$N^{\tg}$ for some simple weight $\sl_2$-module $N$;
\item\label{thm-4.2}
If $V$ has a nonzero central charge, say $\dot{z}\in \C^*$, then one of the following holds:
\begin{enumerate}[$($a$)$]
\item\label{thm-4.2.1} $V$ is a simple highest or lowest weight module;
\item\label{thm-4.2.2} $l=\frac{1}{2}$ and $V$ is isomorphic to
$D(a,\dot{z})^{\tg^{({1}/{2})}}\otimes N^{\tg^{({1}/{2})}}$ for some
$a\in\C\setminus\Z$ and a finite dimensional simple $\sl_2$-module $N$.
\end{enumerate}
\end{enumerate}
\end{theorem}

In the case $l=1/2$ a disguised version of Theorem~\ref{thm-4} is proved in \cite{D}.

The paper is organized as follows. In Section~\ref{s1}, we prove Theorems~\ref{thm-1}, \ref{thm-2}  
and \ref{thm-3}. Using Theorem~\ref{thm-1}, we construct three examples of families of simple $\tg^{(l)}$-modules: 
the first one containing weight modules with finite dimensional weight spaces, the second one containing weight 
modules with infinite dimensional weight spaces, and the third one containing non-weight modules. Given simple
$\tH^{(l)}$-modules with nonzero central charge and simple $\sl_2$-modules, Theorem~\ref{thm-3} provides a 
lot new examples of simple $\tg$-mo\-dules. 

In Section~\ref{s2} we prove Theorems~\ref{highestN}, \ref{thm-4-1}  and \ref{thm-4}. We use
our oscillator representations of $\tg^{(l)}$ and classification of all  simple graded modules 
with a finite dimensional homogeneous spaces over finite dimensional Heisenberg algebras. 
We would like to mention that our results on simple graded modules over finite Heisenberg algebras 
are somewhat similar to those for infinite Heisenberg algebras but our proofs here are quite different 
(compare e.g. with \cite{F}).

\section{Oscillator representations}\label{s1}

In this section, we will prove Theorems~\ref{thm-1}, \ref{thm-2}  and \ref{thm-3}.

\subsection{Proof of  Theorem  \ref{thm-1}}\label{s1.1}
Claim \eqref{thm-1.2} follows directly from claim  \eqref{thm-1.1}. To prove claim \eqref{thm-1.1}
we need to verify the defining relations (i.e. that Lie bracket) for the images (under $\Phi$) of the generators. 
If both generators belong to the Heisenberg subalgebra, there is nothing to check. For the remaining generators
the check is a direct but lengthy computation which we organize in a series of lemmata  below.

\begin{lemma}\label{claim1} 
We have $[H,p_i]=2(l-i)p_i$ for all $i=0,1,\ldots,2l$.
\end{lemma}

\begin{proof}
Using the fact that $2l$ is odd, we have
$$\aligned{ }
[H,p_i]&=[\frac{2}{{z}}\left(\sum_{k=0}
^{l-\frac{1}{2}}(-1)^{k+l-\frac{1}{2}}\frac{l-k}{k!(2l-k)!}p_{2l-k}p_k\right)
-\frac{(l+\frac{1}{2})^2}{2},p_i]\\
&=[\frac{1}{{z}}\left(\sum_{k=0}^{2l}(-1)^{k+l-\frac{1}{2}}\frac{l-k}{k!(2l-k)!}p_{2l-k}p_k\right),p_i]
\endaligned$$
$$\aligned{ }
&=\frac{1}{{z}}\left(\sum_{k=0}^{2l}(-1)^{k+l-\frac{1}{2}}\frac{l-k}{k!(2l-k)!}[p_{2l-k}p_k,p_i]\right)\\
&=\frac{1}{{z}}\left(\sum_{k=0}^{2l}(-1)^{k+l-\frac{1}{2}}\frac{l-k}{k!(2l-k)!}[p_{2l-k},p_i]p_k\right)\\
&\hskip 10pt+\frac{1}{{z}}\left(\sum_{k=0}^{2l}(-1)^{k+l-\frac{1}{2}}\frac{l-k}{k!(2l-k)!}p_{2l-k}[p_k,p_i]\right)\\
&=\frac{1}{{z}}(-1)^{i+l-\frac{1}{2}}\frac{l-i}{i!(2l-i)!}(-1)^{2l-i+l+\frac{1}{2}}i!(2l-i)!{z}p_i\\
&\hskip 10pt +\frac{1}{{z}}(-1)^{2l-i+l-\frac{1}{2}}\frac{l-(2l-i)}{i!(2l-i)!}(-1)^{2l-i+l+\frac{1}{2}}i!(2l-i)!{z}p_{i}\\
&=2(l-i)p_i.\endaligned$$
\end{proof}

\begin{lemma}\label{claim2}
We have $[E,p_i]=ip_{i-1}$ and $[F,p_i]=(2l-i)p_{i+1}$ for all $i=0,1,\ldots,2l$.
\end{lemma}

\begin{proof}
We first verify the relation  $[E,p_i]=ip_{i-1}$:
$$\aligned{ }&[E,p_i]=[\frac{1}{{z}}\left(\sum_{k=1}^{2l}(-1)^{k+l+\frac{1}{2}}
\frac{(l-k)}{(k-1)!(2l-k)! }p_{k-1}p_{2l-k}\right),p_i]\\
&=\frac{1}{{z}}\left(\sum_{k=1}^{2l}(-1)^{k+l+\frac{1}{2}}\frac{(l-k)}{(k-1)!(2l-k)! }[p_{k-1}p_{2l-k},p_i]\right)\\
&=\frac{1}{{z}}\left(\sum_{k=1}^{2l}(-1)^{k+l+\frac{1}{2}}\frac{(l-k)}{(k-1)!(2l-k)! }[p_{k-1},p_i]p_{2l-k}\right)\\
&\hskip 10pt+\frac{1}{{z}}\left(\sum_{k=1}^{2l}(-1)^{k+l+\frac{1}{2}}\frac{(l-k)}{(k-1)!(2l-k)! }p_{k-1}[p_{2l-k},p_i]\right)\\
&=\frac{(-1)^{2l-i+1+l+\frac{1}{2}}(l-(2l-i+1))}{{z}(2l-i+1-1)!(2l-(2l-i+1))! }(-1)^{l+2l-i+\frac{1}{2}}i!(2l-i)!{z}p_{i-1}\\
&\hskip 10pt+\frac{1}{{z}}(-1)^{i+l+\frac{1}{2}}\frac{(l-i)}{(i-1)!(2l-i)! }(-1)^{l+2l-i+\frac{1}{2}}i!(2l-i)!{z}p_{i-1}\\
&=(-(i-l-1)i-i(l-i))p_{i-1}\\
&=ip_{i-1}.\endaligned$$

Now we verify the relation  $[F,p_i]=(2l-i)p_{i+1}$:
$$\aligned{ }&[F,p_i]=\frac{1}{{z}}\left(\sum_{k=1}^{2l}(-1)^{k+l-\frac{1}{2}}
\frac{(l-k)}{(k-1)!(2l-k)! }[p_{k}p_{2l-k+1},p_i]\right)\\
&=\frac{1}{{z}}\left(\sum_{k=1}^{2l}(-1)^{k+l-\frac{1}{2}}\frac{(l-k)}{(k-1)!(2l-k)! }[p_{k},p_i]p_{2l-k+1}\right)\\
\endaligned$$
$$\aligned{ }
&\hskip 10pt +\frac{1}{{z}}\left(\sum_{k=1}^{2l}(-1)^{k+l-\frac{1}{2}}\frac{(l-k)}{(k-1)!(2l-k)!
}[p_{2l-k+1},p_i]p_{k}\right) \\
&=\frac{(-1)^{2l-i+l-\frac{1}{2}}(l-(2l-i))}{{z}(2l-i-1)!(2l-(2l-i))!
}
(-1)^{l+2l-i+\frac{1}{2}}i!(2l-i)!{z}p_{i+1}\\
&\hskip 10pt +\frac{(-1)^{i+1+l-\frac{1}{2}}(l-i-1)}{{z}(i+1-1)!(2l-i-1)! }(-1)^{l+2l-i+\frac{1}{2}}i!(2l-i)!{z}p_{i+1} \\
&=(2l-i)p_{i+1}.\endaligned$$
\end{proof}

\begin{lemma}\label{claim3} 
We have $[E,F]=H$.
\end{lemma}

\begin{proof}
We will use Lemma~\ref{claim2} in the following computation:
$$\aligned{ }&[E,F]=[E, \frac{1}{{z}}\sum_{k=1}^{2l}(-1)^{k+l-\frac{1}{2}}\frac{(l-k)}{(k-1)!(2l-k)! }p_{k}p_{2l-k+1}]\\
&=\frac{1}{{z}}\sum_{k=1}^{2l}(-1)^{k+l-\frac{1}{2}}\frac{(l-k)}{(k-1)!(2l-k)! }[E, p_k]p_{2l-k+1}\\
&\hskip 10pt +\frac{1}{{z}}\sum_{k=1}^{2l}(-1)^{k+l-\frac{1}{2}}\frac{(l-k)}{(k-1)!(2l-k)! }p_{k}[E,p_{2l-k+1}]\\
&=\frac{1}{{z}}\sum_{k=1}^{2l}(-1)^{k+l-\frac{1}{2}}\frac{(l-k)k}{(k-1)!(2l-k)! }p_{k-1}p_{2l-k+1}\\
&\hskip 10pt +\frac{1}{{z}}\sum_{k=1}^{2l}(-1)^{k+l-\frac{1}{2}}\frac{(l-k)(2l-k+1)}{(k-1)!(2l-k)! }p_{k}p_{2l-k}\\
&=\frac{1}{{z}}\sum_{k'=0}^{2l-1}(-1)^{k'+1+l-\frac{1}{2}}\frac{(l-k'-1)(k'+1)}{(k')!(2l-k'-1)! }p_{k'}p_{2l-k'}\\
&\hskip 10pt +\frac{1}{{z}}\sum_{k=1}^{2l}(-1)^{k+l-\frac{1}{2}}\frac{(l-k)(2l-k+1)}{(k-1)!(2l-k)! }p_{k}p_{2l-k}\\
&=\frac{1}{{z}}\sum_{k'=0}^{2l-1}(-1)^{k'+1+l-\frac{1}{2}}\frac{(l-k'-1)(k'+1)(2l-k')}{(k')!(2l-k')! }p_{k'}p_{2l-k'}\\
&\hskip 10pt +\frac{1}{{z}}\sum_{k=0}^{2l}(-1)^{k+l-\frac{1}{2}}\frac{(l-k)(2l-k+1)k}{k!(2l-k)!
}p_{k}p_{2l-k}\\
&=\frac{1}{{z}}\sum_{k=0}^{2l}(-1)^{k+l+\frac{1}{2}}\frac{(k+1)(l-k-1)(2l-k)+(l-k)(-2l+k-1)k}{k!(2l-k)!}p_kp_{2l-k}\\
&= \frac{1}{{z}}\sum_{k=0}^{l-\frac{1}{2}}(-1)^{k+l+\frac{1}{2}}\frac{3\,{k}^{2}+2\,{l}^{2}-6\,kl-2\,l+k}{k!(2l-k)!}p_kp_{2l-k}
\endaligned$$
$$\aligned{ }
&\hskip 10pt +\frac{1}{{z}}\sum_{k=l+\frac{1}{2}}^{2l}(-1)^{k+l+\frac{1}{2}}\frac{3\,{k}^{2}+2\,{l}^{2}-6\,kl-2\,l+k}{k!(2l-k)!}p_kp_{2l-k}\\
&=\frac{1}{{z}}\sum_{k=0}^{l-\frac{1}{2}}(-1)^{k+l+\frac{1}{2}}\frac{3\,{k}^{2}+2\,{l}^{2}-6\,kl-2\,l+k}{k!(2l-k)!}p_{2l-k}p_k\\
&\hskip 10pt +\frac{1}{{z}}\sum_{k=0}^{l-\frac{1}{2}}(-1)^{k+l+\frac{1}{2}}\frac{3\,{k}^{2}+2\,{l}^{2}-6\,kl-2\,l+k}{k!(2l-k)!}
(-1)^{l+k+\frac{1}{2}}k!(2l-k)!{z}\\
&\hskip 10pt +\frac{1}{{z}}
\sum_{k=l+\frac{1}{2}}^{2l}(-1)^{k+l+\frac{1}{2}}\frac{3\,{k}^{2}+2\,{l}^{2}-6\,kl-2\,l+k}{k!(2l-k)!}p_kp_{2l-k}\\
&=\frac{1}{{z}}\sum_{k=0}^{l-\frac{1}{2}}(-1)^{k+l+\frac{1}{2}}\frac{3\,{k}^{2}+2\,{l}^{2}-6\,kl-2\,l+k
}{k!(2l-k)!}p_{2l-k}p_k\\
&\hskip 10pt +\sum_{k=0}^{l-\frac{1}{2}}(3\,{k}^{2}+2\,{l}^{2}-6\,kl-2\,l+k)\\
&\hskip 10pt +\frac{1}{{z}}\sum_{k=l+\frac{1}{2}}^{2l}(-1)^{k+l+\frac{1}{2}}
\frac{3\,{k}^{2}+2\,{l}^{2}-6\,kl-2\,l+k}{k!(2l-k)!}p_kp_{2l-k}\\
&=\frac{1}{{z}}\sum_{k=0}^{l-\frac{1}{2}}(-1)^{k+l+\frac{1}{2}}
\frac{3\,{k}^{2}+2\,{l}^{2}-6\,kl-2\,l+k}{k!(2l-k)!}p_{2l-k}p_k-\frac{(l+\frac{1}{2})^2}{2}\\
&\hskip 10pt +\frac{1}{{z}}\sum_{k=0}^{l-\frac{1}{2}}(-1)^{2l-k+l
+\frac{1}{2}}\frac{-6\,kl+3\,{k}^{2}-k+2\,{l}^{2}}{k!(2l-k)!}p_{2l-k}p_{k}\\
&=\frac{1}{{z}}\sum_{k=0}^{l-\frac{1}{2}}(-1)^{k+l+\frac{1}{2}}
\frac{2(k-l)}{k!(2l-k)!}p_{2l-k}p_k-\frac{(l+\frac{1}{2})^2}{2}\\
&=
\frac{1}{{z}}\sum_{k=0}^{l-\frac{1}{2}}(-1)^{k+l+\frac{1}{2}}\frac{2(k-l)}{k!(2l-k)!}p_{2l-k}p_k-\frac{(l+\frac{1}{2})^2}{2}=H.
\endaligned$$
\nopagebreak
\end{proof}

Finally, the relations $[H, E]=2E$ and $[H, F]=-2F$ follow easily from Lemma~\ref{claim1}. 
Theorem \ref{thm-1} follows.

\subsection{Proof of  Theorem~\ref{thm-2}}\label{s1.2}

Theorem~\ref{thm-2} can be proved by
a direct computation. We give here a more conceptual argument.

Let $w$ be a highest weight
vector of $M_{\tg}(\dot{z},\dot{h})$ and assume $\dot{z}\ne 0$. Then
$U(\tH) w=M_{\tH}(\dot{z})$, which is a simple $\tH$-module. From
Theorem \ref{thm-1},  we have the simple highest weight $\tg$ module
$(U(\tH)w)^{\tg}$ with $h\cdot w=-\frac{(l+\frac{1}{2})^2}{2}w$ and
$e\cdot w=0$, where we use $\cdot$ to denote the new action. Hence
we have the decomposition $U(\tH) w\cong (U(\H)w)^{\tg} \otimes (\C u)$ of 
$\C h+\tH+\C e$-modules, 
where $\C u$ is regarded as a  $\C h+\C e+\tH$-module via $h
u=(\dot{h}+\frac{(l+\frac{1}{2})^2}{2}) u$ and $(\C e+\tH)u=0$. Using
\cite[Lemma~8]{LZ1}, we have
$$\aligned M_{\tg}(\dot{z},\dot{h})&\cong \Ind_{\C h+\C
e+\tH}^{\tg}\Ind_{\tg_++\tg_0}^{\C h+\C e+\tH} \C w\cong \Ind_{\C
h+\C e+\tH}^{\tg} (U(\tH) w)\\ &\cong  \Ind_{\C h+\C e+\tH}^{\tg}
((M_{\tH}(\dot{z}))^{\tg} \otimes \C u)\cong
(M_{\tH}(\dot{z}))^{\tg}\otimes \Ind_{\C h+\C e+\tH}^{\tg} \C u\\
&\cong M_{\tH}(\dot{z})^{\tg}\otimes
M_{\mathfrak{sl}_2}(\dot{h}+\text{\tiny$\frac{1}{2}$}(l+\text{\tiny$\frac{1}{2}$})^2)^{\tg}.\endaligned$$
The second part of Theorem~\ref{thm-2} follows from this and
\cite[Theorem~7]{LZ1}. 

\subsection{Proof of  Theorem  \ref{thm-3}}\label{s1.3}

Claim~\eqref{thm-3.1} follows from \cite[Theorem~7]{LZ1}.
To prove claim~\eqref{thm-3.2}, let $\tH=\tH^{(l)}$. We have an isomorphism  
$M\cong M^{\tg}\otimes \C$ of $\tH$-modules, where $\C$
 is the trivial $\tH$-module. Applying \cite[Lemma~8]{LZ1}, we get
$$\Ind_{\tH}^{\tg} M\cong \Ind_{\tH}^{\tg} (M^{\tg}\otimes \C)\cong M^{\tg}\otimes \Ind_{\tH}^{\tg}
\C\cong  M^{\tg}\otimes U(\sl_2)^{\tg}.$$

Now $V$ is isomorphic to a simple quotient of  $M^{\tg}\otimes
U(\sl_2)^{\tg}$, which, by \cite[Theorem~7]{LZ1}, has to be of
the form $M^{\tg}\otimes N^{\tg}$ for some simple $\sl_2$-module
$N$. 

\subsection{Examples}\label{s1.4}

Here we give three concrete examples of simple $\tg^{(l)}$-mo\-du\-les with different properties.

\begin{example}
{\rm 
Let $\dot{z}\in \C^*$, $l\in N-\frac{1}{2}$, $\tg=\tg^{(l)}$ and
$F=\C[x_1, x_3,\ldots, x_{2l}]$ be the usual Fock space . Then
we have the classical oscillator representation of $\tg=\tg^{(l)}$ on
$F$ with the action given as follows (here $\rightarrow$ means ``acts as''): 
$$ p_k\rightarrow x_{2(k-l)}, \quad k=l+\frac{1}{2},\ldots, 2l,$$
$$ p_k \rightarrow (-1)^{k+l+\frac{1}{2}}\dot{z}k!(2l-k)!\frac{\partial}{\partial x_{2(l-k)}}, \quad k=0,1,\ldots,l-\frac{1}{2},$$
$$e\rightarrow -\frac{\dot{z}}{2}\left((l+\frac{1}{2})!\frac{\partial}{\partial x_1}\right)^2+
\sum_{k=1}^{l-\frac{1}{2}}(2l-k+1)x_{2(l-k)}\frac{\partial}{\partial
x_{2(l-k+1)}},$$
$$  f\rightarrow \frac{1}{2\dot{z}}{\left(\frac{x_1}{(l-\frac{1}{2})!}\right)}^2+\sum_{k=0}^{l-\frac{1}{2}} kx_{2(l-k+1)}
\frac{\partial}{\partial x_{2(l-k)}},$$
$$ h\rightarrow -d-\frac{(l+\frac{1}{2})^2}{2},\quad z\rightarrow\dot{z};$$ 
where $d=\sum_{k=0}^{l-\frac{1}{2}}2(l-k)x_{2(l-k)}\frac{\partial}{\partial x_{2(l-k)}}$ is the degree derivation.
It is easy to check that $F\cong
\bar{M}_{\tg}(\dot{z},-\frac{(l+\frac{1}{2})^2}{2})$.
}
\end{example}

Now we construct simple $\tg$-module from simple Whittaker modules over $\tH$ (in the sense of \cite{BM,C}).

\begin{example} 
{\rm
Let $\dot{z}\in \C^*$, $l\in N-\frac{1}{2}$, $\tg=\tg^{(l)}$ and $F=\C[x_1, x_3,\ldots, x_{2l}]$. Then for any $\mu=(\mu_0,\mu_1,\ldots,\mu_{l-\frac{1}{2}})\in \C^{l+\frac{1}{2}}$ we have the oscillator representation of $\tg=\tg^{(l)}$ on $F$ with the action given by $z\rightarrow\dot{z}$,
$$ p_k\rightarrow x_{2(k-l)}, \quad k=l+\frac{1}{2},\ldots, 2l,$$
$$ p_k \rightarrow (-1)^{k+l+\frac{1}{2}}\dot{z}k!(2l-k)!\left(\frac{\partial}{\partial x_{2(l-k)}}+\mu_k\right), \quad k=0,1,\ldots,l-\frac{1}{2},$$
$$e\rightarrow -\frac{\dot{z}}{2}\left((l+\frac{1}{2})!(\frac{\partial}{\partial x_1}+\mu_{l-\frac{1}{2}})\right)^2+\sum_{k=1}^{l-\frac{1}{2}}(2l-k+1)x_{2(l-k)}\left(\frac{\partial}{\partial x_{2(l-k+1)}}+\mu_{k-1}\right),$$
$$  f\rightarrow \frac{1}{2\dot{z}}{\left(\frac{x_1}{(l-\frac{1}{2})!}\right)}^2+\sum_{k=0}^{l-\frac{1}{2}} kx_{2(l-k+1)}\left(\frac{\partial}{\partial x_{2(l-k)}}+\mu_k\right),$$
$$ h\rightarrow -d-\left(\sum_{k=0}^{l-\frac{1}{2}}2(l-k)x_{2(l-k)}\mu_k\right)-\frac{(l+\frac{1}{2})^2}{2},$$ where $d$ is the degree derivation from the previous example.
This simple $\tg$-module on $F$ is isomorphic to the module
$(\Ind_{\tH^++\C z}^{\tH}\C w)^{\tg}$ where $zw=\dot zw$ and
$p_iw=(-1)^{i+l+1/2}i!(2l-i)!\mu_iw$ for $i=0, 1, ..., l-1/2$.
}
\end{example}

The next example is constructed using simple weight $\tH$-modules with
infinite dimensional weight spaces.

\begin{example} 
{\rm
Let $\dot{z}\in \C^*$, $l\in N-\frac{1}{2}$,  $\tg=\tg^{(l)}$ and $\hat{F}=\C[x_1^{\pm 1}, x_3^{\pm 1},\ldots, x_{2l}^{\pm 1}]$. Then for any $\mu=(\mu_0,\mu_1,\ldots,\mu_{l-\frac{1}{2}})\in \C^{l+\frac{1}{2}}$, we have the oscillator representation of $\tg=\tg^{(l)}$ on $\hat{F}$ with the action given by $z\rightarrow\dot{z}$,
$$ p_k\rightarrow x_{2(k-l)}, \quad k=l+\frac{1}{2},\ldots, 2l,$$
$$ p_k \rightarrow (-1)^{k+l+\frac{1}{2}}\dot{z}k!(2l-k)!\left(\frac{\partial}{\partial x_{2(l-k)}}+\mu_kx_{2(l-k)}^{-1}\right),
\quad k=0,1,\ldots,l-\frac{1}{2},$$
$$e\rightarrow -\frac{\dot{z}}{2}\left((l+\frac{1}{2})!(\frac{\partial}{\partial x_1}+\mu_{l-\frac{1}{2}}x_1^{-1})\right)^2+\sum_{k=1}^{l-\frac{1}{2}}(2l-k+1)x_{2(l-k)}\left(\frac{\partial}{\partial x_{2(l-k+1)}}+\mu_{k-1}x_{2(l-k+1)}^{-1}\right),$$
$$  f\rightarrow \frac{1}{2\dot{z}}{\left(\frac{x_1}{(l-\frac{1}{2})!}\right)}^2+\sum_{k=0}^{l-\frac{1}{2}} kx_{2(l-k+1)}\left(\frac{\partial}{\partial x_{2(l-k)}}+\mu_kx_{2(l-k)}^{-1}\right),$$
$$ h\rightarrow -d-\left(\sum_{k=0}^{l-\frac{1}{2}}2(l-k)\mu_k\right)-\frac{(l+\frac{1}{2})^2}{2},$$ where $d$ is the degree derivation from the previous examples.
It is clear that $\hat{F}$ is a weight $\tg$-module and all nonzero weight spaces of $\hat{F}$ 
are infinite dimensional for $l>\frac{1}{2}$.  It is easy to check that $\hat{F}$ is simple 
if and only if $\mu_i\notin \Z$ for all $i=0,1,\ldots, l-\frac{1}{2}$.
}
\end{example}

Similarly to the above examples,  using Theorem \ref{thm-3} one constructs a lot of simple
$\tg^{(l)}$-modules from   simple $\sl_2$-modules and simple
$\tH^{(l)}$-modules with nonzero central charge. All simple modules over $\sl_2$ and  $\tH^{(1/2)}$ are
described in \cite{B}. The problem to classify all simple modules  over $\tH^{(l)}$ 
with $l\in \frac{1}{2}+\N$ is still open.

\section{Simple weight modules with a finite dimensional nonzero weight space}\label{s2}

In this section, we will classify all simple $\tg^{(l)}$-modules and all simple $\g^{(l)}$-modules
which are weight modules and have a finite dimensional nonzero weight space. 

\subsection{Proof of Theorem~\ref{highestN}}\label{s2.1}

Claim \eqref{highestN.1} is straightforward. To prove claim \eqref{highestN.2} assume that $\dot{p}_{l}\ne 0$.
It is easy to check that the $\mathfrak{n}$-module $F_1$ defined by \eqref{eqaa1} and \eqref{eqaa2} is a
simple  highest weight $\mathfrak{n}$-module. Now claim \eqref{highestN.2} would follow if we could prove that
$\Ind_{\mathfrak{n}}^\g F_1$ is a simple  $\g$-module.

Assume that $\Ind_{\mathfrak{n}}^\g F_1$ is not simple and therefore there exists some  nonzero
$v=\sum_{i=0}^k f^k\otimes v_i$, where $v_i\in F_1$ with $U(\g) v\ne \Ind_{\mathfrak{n}}^\g F_1$. 
We have $k>0$ and we may assume that $k$ is minimal. Recall that $p_i F_1=0$ for all $i=0,\ldots, l-1$. It
is now easy to see that $0\ne p_{l-1} v=\sum_{i=0}^k ([p_{l-1}, f^k])\otimes
v_i$ has degree $k-1$ with respect to $f$, a contradiction.

\subsection{Preliminary results on weight modules}\label{s2.2}

In this subsection we extend the techniques and methods from \cite{D} to all conformal Galilei algebras.

\begin{lemma}\label{local-nilpotency} 
Let $l\in \frac{1}{2}\N$ and $s\in \{e,f,p_k\,\,|\,\,k=0,\ldots,2l\}\subset \g=\g^{(l)}$. 
\begin{enumerate}[$($i$)$]
\item\label{local-nilpotency.1} The action of $\ad s$ on $U(\g)$ is locally nilpotent.
\item\label{local-nilpotency.2} If $V$ is a $\g$ module, then the set
$\{v\in V\,\,|\,\,s^n v=0 \,\,\rm{for\,\,some\,\,}n\in \N\}$ is a submodule of $V$.
\item\label{local-nilpotency.3} Let $V$ be  a simple $\g$ module. Then $s$ acts locally
nilpotently on $V$ if and only if there exists some $0\ne v\in V$
such that $s v=0$.
\end{enumerate}
\end{lemma}

\begin{proof} Follows mutatis mutandis the proof of \cite[Lemma~3]{D}. \end{proof}

\begin{lemma}\label{lemma-z=0} 
Let $l\in \frac{1}{2} \N$, and $V$ be a simple weight $\g^{(l)}$-module. Suppose that 
$\supp(V)\subseteq \lambda+\Z$ for some $\lambda\in\C$, that $\dim V_{\lambda}<\infty$ and that $\H  V\ne 0$.
\begin{enumerate}[$($i$)$]
\item\label{lemma-z=0.1} If $V$ is not a highest (resp. lowest) weight module,
then $e$ (resp. $f$) acts injectively on $V$.
\item\label{lemma-z=0.2} 
If $V$ is neither a highest nor a lowest weight module,  then both
$e$ and $f$ act bijectively on $V$.  Consequently,  for any $\mu\in \supp(V)$,   $\dim
V_{\mu+2i}=\dim V_{\mu}$ for all $i\in \Z$. If   $l\in \N-\frac{1}{2}$, then $\dim V_{\mu+i}=\dim
V_{\mu}$ for all $i\in \Z$. 
\end{enumerate}
\end{lemma}

\begin{proof} 
We start with claim \eqref{lemma-z=0.1}. We prove the claim for $e$ (and then the claim for $f$ follows
by symmetry). Assume that $V$ is not a highest weight module and that the kernel of $e$ on $V$ is nonzero.
Consider first the situation when the kernel of each $p_i$, $i=0,1,\ldots,\lfloor l-1/2\rfloor$,
on $V$ is nonzero. Using Lemma \ref{local-nilpotency} and nilpotency of the Lie algebra 
$\span\{e,p_i\,|\,i=0,\ldots,\lfloor l-1/2\rfloor\}$, we may find a common eigenvector for 
all these elements. This implies that $V$ is a highest weight module, a contradiction. 
Therefore such situation is not possible. Hence there is a minimal $k\in \{0,1,\ldots,\lfloor l-1/2\rfloor\}$
such that the element $p_k$ acts injectively on $V$. Again, using Lemma \ref{local-nilpotency} and nilpotency of the Lie algebra $\span\{e, p_0,\ldots, p_{k-1}\}$, we may find a weight vector $0\ne w\in V_{\lambda+k_0}$ with $e w=p_i w=0$ for all $i=0,1,\ldots, k-1$. Now it is straightforward to verify that for each $j\in \N$ the vector
$p_k^j w$ is a nonzero highest weight vector for the $\sl_2$-subalgebra. This implies that all weight spaces 
of $V$ are infinite dimensional, a contradiction. Claim \eqref{lemma-z=0.1} follows.

Now we prove claim \eqref{lemma-z=0.2}. From claim \eqref{lemma-z=0.1}, both $e$ and $f$ act injectively on $V$. 
This implies that $\dim V_{\lambda+i}=\dim V_{\lambda+j}$ for all $i-j\in 2\Z$, in particular, 
both $e$ and $f$ act bijectively on $V$.
If $l\in \N-\frac{1}{2}$ and $\dim V_{\lambda}\ne \dim
V_{\lambda+1}$, then  for all $i=0,1,\ldots,2l$ the kernel of $p_i$ is nonzero.
Similarly to the above it follows that  $\H w=0$ for some
$w\in V$, a contradiction. Hence $\dim V_{\lambda+i}=\dim
V_{\lambda+j}<\infty$ for all $i,j\in \Z$. This completes the proof.
\end{proof}

For $l\in \frac{1}{2}\N$ and $\g=\g^{(l)}$ we denote by $U(\g)^{(f)}$ the localization of $U(\g)$ 
with respect to the multiplicative set $\{f^i\vert i\in\Z_+\}$.

\begin{lemma}\label{auto} 
Let $l\in \frac{1}{2}\N$ and $\g=\g^{(l)}$. For any $x\in \C$  there is a unique 
automorphism $\theta_x$ of $U(\g)^{(f)}$ such that
\begin{displaymath}
\theta_x(f)=f,\quad \theta_x(h)=h-2x ,\quad\theta_x(e)=e+x(h-1-x)f^{-1}
\end{displaymath}
and
\begin{equation}
\label{def-theta2}
\theta_x(p_{j})=\sum_{k=0}^{2l-j}(-1)^k \binom 
xk\frac{(2l-j)!}{(2l-j-k)!}f^{-k}p_{j+k},\quad  j=0,1,\ldots,2l.
\end{equation}
\end{lemma}

\begin{proof}  
The proof follows mutatis mutandis the proof of \cite[Lemma~4.3]{Mt} (see also \cite[Proposition~8]{D} or 
\cite[Proposition~3.45]{M}).
\end{proof}

Let $\g=\g^{(l)}$ and $V$ be a $\g$-module on which $f$ act bijectively. Let $\sigma$ be an automorphism 
of $U(\g)^{(f)}$. Setting $g \cdot v=\sigma(g) v$  for all $g\in \g$ and $v\in V$ defines on $V$ a 
new  $\g$-module structure which we will denote by $V^{\sigma}$.

\begin{proposition}\label{thm-4-z=0} 
Let $l\in \frac{1}{2}\N$, and $V$ be a simple weight $\g^{(l)}$-module with a nonzero finite dimensional 
weight space. If $V$ is not a highest or a lowest weight module, then $\H V=0$.
\end{proposition}

\begin{proof} 
Suppose that $V$ is neither a highest nor a lowest weight module and $\H V\ne0$. Let $\lambda\in \supp(V)$ with $0\ne\dim V_{\lambda}=n<\infty$.  Then from Lemma~\ref{lemma-z=0} we see that $e$ and $f$ acts bijectively on $V$ and $V$ is uniformly bounded. Take an eigenvector $v_{\lambda}\in V_{\lambda}$  of $ef$. Then it is easy to see that there exists some $x\in \C$ such that $\theta_x(e)fv_{\lambda}=(e+x(h-1-x)f^{-1})fv_{\lambda}=0$. Now we consider the $\g$-module $V^{\theta_x}$ which is uniformly bounded weight module. From Lemma \ref{local-nilpotency}\eqref{local-nilpotency.2}
we have that 
\begin{displaymath}
M=\{v\in V^{\theta_x}|e\,\text{  acts locally nilpotently on }\,v \} 
\end{displaymath}
is a nonzero $\g$-submodule of $V^{\theta_x}$. Moreover, being a uniformly bounded module, $M$ has 
finite length by \cite[Lemma~3.3]{Mt} and hence it has a simple submodule $M'$ (which is uniformly bounded as well). If $l\in \N$, from Proposition \ref{highestN} it follows that $\theta_x(\H) M'=0$. For $l\in \N-1/2$ the fact that 
$\theta_x(\H) M'=0$ is obvious.Using induction and \ref{def-theta2} we obtain 
$p_{2l-i}M'=0$ for $i=0,1,\ldots,2l$. Since the subspace annihilated by $\H$ is a $\g^{(l)}$-submodule of $V$, 
we have $\H V=0$, a contradiction. This proves the proposition.
\end{proof}

Next we will first establish   some properties for weight $\C h+\tH^{(l)}$-modules 
with nonzero central charge. Choose a new basis of $\C h+\tH^{(l)}$ as follows:
\begin{displaymath}
t_{2(l-k)}:=
\begin{cases}
(-1)^{k+l+\frac{1}{2}}2(l-k)\frac{p_k}{k!},& k=0,1,\ldots,l-\frac{1}{2};\\
\frac{p_k}{k!},& k=l+\frac{1}{2},\ldots,2l.
\end{cases}
\end{displaymath}
Then $\C h+\tH^{(l)}=\span\{h, t_i, z| i=\pm 1,\pm 3,\ldots, \pm
2l\}$ with
\begin{displaymath}[t_i,t_j]=i\delta_{i+j,0} z,\quad [h,t_i]=i t_i, \quad i, j=\pm 1,\ldots,\pm 2l.
\end{displaymath}
We may naturally regard $\C h+\tH^{(l)}$  as a subalgebra of $\C h+\tH^{(l+1)}$.

\begin{lemma} \label{heisenberg-1}
Let $l\in \N-\frac{1}{2}$, $V$ be a simple weight $\C h+\tH^{(l)}$-module with nonzero central charge $\dot{z}$ and $\supp(V)\subset \lambda+\Z$ with $\dim V_{\lambda}<\infty$  for some $\lambda\in \C$.
\begin{enumerate}[$($i$)$]
\item\label{heisenberg-1.1} If $l=\frac{1}{2}$, then $V$ is a highest (or a lowest) weight
module or isomorphic to  $D(a,\dot{z})$ for some $a\in\C$.
\item\label{heisenberg-1.2} If $l\in \N+\frac{1}{2}$, then $V$ is either a highest or a
lowest weight module.
\end{enumerate}
\end{lemma}

\begin{proof} 
We start with claim \eqref{heisenberg-1.1}. Let $l=\frac{1}{2}$, $\tH=\tH^{(\frac{1}{2})}$ and $M$ be a simple 
weight $\C h+\tH^{(l)}$-module with central change $\dot{z}$. For $\lambda\in \supp(M)$ the space $M_{\lambda}$ 
is a simple $\C(t_1t_{-1},h,z)/(h-\lambda,z-\dot{z})\cong \C[t_1t_{-1}]$-module. As $\C[t_1t_{-1}]$ is commutative,
it follows that $\dim V_{\lambda}=1$. As usual for weight modules (see e.g. \cite{DOF}), there is exactly one 
(up to isomorphism) simple weight $\C h+\tH^{(l)}$-module $N$ such that $N_{\lambda}\cong M_{\lambda}$ 
as $\C(t_1t_{-1},h,z)$-modules (this module is, in fact, $M$). It is easy to check that in the case $\dot{z}\neq 0$ 
all simple $\C(t_1t_{-1},h,z)$-modules can be obtained  restricting simple highest weight $\C h+\tH^{(l)}$-modules,
simple lowest weight $\C h+\tH^{(l)}$-modules and  modules $D(a,\dot{z})$ to some nonzero weight spaces. Therefore, 
if $t_1$ does not act injectively on $H$,  then $H$ is a highest weight module; if $t_{-1}$ does not act injectively 
on $H$,  then $H$ is a lowest weight module; if both $t_1$ and $t_{-1}$ act injectively
and $\dot{z}\ne 0$, we have $H\cong D(a,\dot{z})$. Claim \eqref{heisenberg-1.1} follows.

Now we prove claim \eqref{heisenberg-1.2}. Assume $l\in \N+\frac{1}{2}$ and $V$ is a simple weight 
$\C h+\tH^{(l)}$-module with nonzero central charge $\dot{z}$. Further, let $\lambda\in \C$
be such that  $\supp(V)\subset \lambda+\Z$ and  $\dim V_{\lambda}<\infty$. We claim that $V$ contains a simple 
submodule over the algebra $\mathfrak{a}:=\C h+\C t_{2l}+\C t_{-2l}+\C z$.

Indeed, if $V_{\lambda}=0$, it is clear that $V$ contains a highest or lowest weight  simple 
$\C h+\C t_1+\C t_{-1}+\C z$- submodule. Therefore, without loss of generality, we may assume that
either $V_{\lambda}=0$ and $\lambda-\N\subset \supp(V)$, in which case $V$ contains
a highest weight $\mathfrak{a}$-submodule; or
$V_{\lambda}=0$ and $\lambda+\N\subset \supp(V)$, in which case $V$ contains
a lowest weight $\mathfrak{a}$-submodule.

If $V_{\lambda}\ne 0$, we either have that $U(\mathfrak{a})
V_{\lambda}$ is uniformly bounded or that it has a highest or a lowest
weight vector as a $\mathfrak{a}$-module. In both
cases, $V$ has a simple $\mathfrak{a}$-submodule.

Now we are going to prove claim \eqref{heisenberg-1.2} of the lemma by induction on $l$
(we use claim \eqref{heisenberg-1.1} as the basis of the induction).

Define the Lie algebra $\mathfrak{t}=\mathfrak{t}^{(l)}=\span\{h,t_{\pm
i},z_1,z_2|i=1,3,\ldots,2l\}$ with the following Lie bracket:
\begin{displaymath} [z_1,\mathfrak{t}]=[z_2,\mathfrak{t}]=0;\end{displaymath}
\begin{displaymath} [h,t_i]=i t_i, \quad i=\pm 1,\ldots,\pm 2l;\end{displaymath}
\begin{displaymath}[t_i,t_j]=i\delta_{i+j,0} z_1, [h,t_i]=i t_i, \quad i,j=\pm 1,\ldots,\pm 2(l-1);\end{displaymath}
\begin{displaymath}[t_{2l},t_j]=2l\delta_{j,-2l} z_2,  \quad j=\pm 1,\ldots,\pm 2l.\end{displaymath}
The algebra  $\mathfrak{t}$ has a subalgebra $\span\{h,t_{\pm
i},z_1|i=1,3,\ldots, 2(l-1)\}$ which is isomorphic to $\tH^{(l-1)}$. Moreover, we have 
$\mathfrak{t}/\C(z_1-z_2)\cong \tH$. Hence we may regard $V$ as a $\mathfrak{t}$-module
with the action $z_1=z_2=\dot{z}$. Let $\tilde{\mathfrak{t}}=\C t_{2l}+\C h+\C t_{-2l}+\C z_2$. 
Then $V$ contains a simple  $\tilde{\mathfrak{t}}$-submodule $M$. Extend $M$ to a 
$\mathfrak{t}$-module by $t_iM=z_1M=0$ for all $i=\pm 1,\ldots, \pm 2(l-1)$. Denote
the resulting module by $M^\mathfrak{t}$. Let $\C w$ be the one dimensional
module over $\tilde{\mathfrak{t}}+\C z_1$ with $\tilde{\mathfrak{t}} w=0$ and $z_1 w=\dot{z} w$. Then
$M\cong M^\mathfrak{t}\otimes \C w$ as $\tilde{\mathfrak{t}}+\C z_1$ module. Therefore $V=U(\mathfrak{t})
M=U(\H) M$ is a quotient of 
$$\Ind_{\tilde{\mathfrak{t}}+\C z_1}^\mathfrak{t} M\cong \Ind_{\tilde{\mathfrak{t}}+\C
z_1}^\mathfrak{t} (M^\mathfrak{t}\otimes \C w) \cong M^\mathfrak{t}\otimes \Ind_{\tilde{\mathfrak{t}}+\C z_1}^\mathfrak{t} \C w.$$
From \cite[Theorem 7]{LZ1} it follows that any simple quotient
of $M^\mathfrak{t}\otimes \Ind_{\tilde{\mathfrak{t}}+\C z_1}^\mathfrak{t} \C w$ is of the form  $M^\mathfrak{t}\otimes
N$, where $N$ is a simple quotient module of  $\Ind_{\tilde{\mathfrak{t}}+\C z_1}^\mathfrak{t}
\C w$. From $\dim V_{\lambda}<\infty$ and the induction hypothesis,
we have either $M$ and $N$ are both highest weight modules or $M$ and $N$
are both lowest weight modules (note that, in particular, the condition $\dim V_{\lambda}<\infty$
excludes the case $N\cong D(a,\dot{z})$). Then $V$ is a highest or a lowest weight
module, which completes the proof.
\end{proof}

\begin{lemma}\label{heisenberg-2}
Suppose that $l\in \N-\frac{1}{2}$  and $V$ is a weight module over $\C h+\tH^{(l)}$ with 
a finite dimensional nonzero weight space and $z$ acts as a nonzero scalar. 
Then $V$ has a simple $\tH^{(l)}$-submodule with all weight spaces finite dimensional. 
\end{lemma}

\begin{proof} 
Let $\lambda\in \supp(V)$ with $0<\dim V_{\lambda}<\infty$.  Let $W_{\lambda}$ be a simple $U(\tH)_0$-submodule of $V_{\lambda}$ and denote $W=U(\tH)W_{\lambda}$, which is a $\C h+\tH$-submodule of $V$. It is clear that any proper $\C h+\tH$-submodule of $W$ has trivial intersection with $W_{\lambda}$. Hence $W$ has a unique maximal submodule $W'$ with $W_{\lambda}'=0$. If $W'=0$, then $W$ is a simple $\C h+\tH$ submodule with a finite dimensional nonzero weight space. Then from Lemma \ref{heisenberg-1} we have that $W$ is also a simple $\tH$ module with finite dimensional weight space. If $W'\ne 0$, then, without losing of generality, we may assume that $\supp(W')\cap (\lambda-\N)\neq \varnothing$. Denote $W''=U(\C h+\tH)(\oplus_{i\in \N}W_{\lambda-i}')$. By Lemma \ref{heisenberg-1}, any simple subquotient of $W''$ is a highest weight module with highest weight contained in $ \lambda-\N$, i.e., $\oplus_{i\in \N}W_{\lambda-i}'$ is a submodule. Therefore $W'$ contains a highest weight vector, which generates a simple $\tH^{(l)}$ submodule.
\end{proof}

\subsection{Proof of Theorems~\ref{thm-4-1} and \ref{thm-4}}\label{s2.3}

Theorem \ref{thm-4-1} and  claim \eqref{thm-4.1} of Theorem  \ref{thm-4} follow from Proposition~\ref{thm-4-z=0}. 
Claim \eqref{thm-4.2} of Theorem  \ref{thm-4}  follows from Lemmata~\ref{heisenberg-1} and
\ref{heisenberg-2} and Theorem~\ref{thm-3}.
\vspace{5mm}

\subsection{Simple weight modules with an infinite dmensional weight space}\label{s2.4}

From Theorems~\ref{thm-4-1} and \ref{thm-4} and our discussion above we immedeately get the following:

\begin{corollary}\label{coroll}
\begin{enumerate}[$($i$)$]
\item\label{coroll.1} Let $l\in \mathbb{N}$, $\mathfrak{g}=\mathfrak{g}^{(l)}$ and $V$ be a simple weight 
$\mathfrak{g}$-module. Assume that $\dim V_{\lambda}=\infty$ for some $\lambda\in \C$. Then 
$\supp(V)=\lambda+2\Z$ and $\dim V_{\lambda+2i}=\infty$ for all $i\in\Z$.
\item\label{coroll.2} Let $l\in \mathbb{N}-\frac{1}{2}$, $\tilde{\mathfrak{g}}=\tilde{\mathfrak{g}}^{(l)}$ 
and $V$ be a simple weight $\mathfrak{g}$-module. Assume that $\dim V_{\lambda}=\infty$ for some $\lambda\in \C$. 
Then  $\supp(V)=\lambda+\Z$ and $\dim V_{\lambda+i}=\infty$ for all $i\in\Z$.
\end{enumerate}
\end{corollary}

Note that in the case $l=\frac{1}{2}$ this statement is proved in \cite{WZ}.
\vspace{5mm}

\noindent
{\bf Acknowledgments.}
The research in this paper was carried out during the visit of the first author 
to University of Waterloo and of the second author to Wilfrid Laurier University. 
K.Z. is partially supported by  NSF of China (Grant
11271109) and NSERC. R.L. is partially supported by NSF of China
(Grant 11371134) and Jiangsu Government Scholarship for Overseas Studies (JS-2013-313). 
V. M. is partially supported by the Swedish Research Council.
R.L.  would like to thank professors Wentang Kuo
and   Kaiming Zhao for sponsoring his visit, and University of Waterloo
for providing excellent working conditions.
V. M. thanks Wilfrid Laurier University for hospitality.

\vspace{2mm}

\noindent  R.L.: Department of Mathematics, Soochow
university, Suzhou 215006, Jiangsu, P. R. China.
Email: {\tt rencail\symbol{64}amss.ac.cn}
\vspace{2mm}

\noindent
V.M.: Department of Mathematics, Uppsala University,
Box 480, SE-751 06, Uppsala, Sweden. Email: {\tt mazor\symbol{64}math.uu.se}
\vspace{2mm}

\noindent K.Z.: Department of Mathematics, Wilfrid
Laurier University, Waterloo, ON, Canada N2L 3C5,  and College of
Mathematics and Information Science, Hebei Normal (Teachers)
University, Shijiazhuang, Hebei, 050016 P. R. China. Email:
{\tt kzhao\symbol{64}wlu.ca}

\end{document}